\documentclass[12pt]{amsart}
\usepackage{xcolor}
\usepackage{mdwlist}
\usepackage{float}
\allowdisplaybreaks
\usepackage{longtable}
\usepackage{makecell}
\usepackage{ltablex}
\usepackage{amsfonts,amssymb,amsmath,textcomp}
\usepackage{enumitem}

\usepackage{tikz}
\usepackage{enumitem}

\usepackage[justification=centering]{caption}
\usepackage{supertabular}
 2

\usepackage[autostyle]{csquotes}
\theoremstyle{plain}
\newtheorem{thm}{Theorem}[section]

\newtheorem{lem}[thm]{Lemma}
\newtheorem{prop}[thm]{Proposition}
\newtheorem{cor}[thm]{Corollary}
\newcommand{\thmref}[1]{Theorem~\ref{#1}}

\newcommand{\corref}[1]{Corollary~\ref{#1}}
\theoremstyle{definition}
\newtheorem{defn}[thm]{Definition}

\newtheorem*{rmk}{Remark}
\newtheorem*{con}{Conjecture}
\numberwithin{equation}{section}
\begin{document}
\baselineskip 6.1mm
\newcommand{\suchthat}{\;\ifnum\currentgrouptype=16 \middle\fi|\;}
\newcommand{\dmid}{\mathrel\Vert}
\title[Non-P\'olya Fields with Large P\'olya Groups]
{Non-P\'olya Fields with Large P\'olya Groups Arising from Lehmer Quintics}
\author{Nimish Kumar Mahapatra}
\author{Prem Prakash Pandey}
\address[Nimish Kumar Mahapatra]{Indian Institute of Science Education and Research, Berhampur, India.}
\email{nimishkm18@iiserbpr.ac.in}

\address[Prem Prakash Pandey]{Indian Institute of Science Education and Research, Berhampur, India.}
\email{premp@iiserbpr.ac.in}


\date{\today}


\begin{abstract}
In this article we construct a new family of quintic non-P\'olya fields with large P\'olya groups. We study the upper bound of P\'olya numbers of such fields and show that the P\'olya numbers never exceed five times the size of its P\'olya group. Finally we show that such non-P\'olya fields are non-monogenic fields of field index one. 
\end{abstract}

\maketitle{}

\section{Introduction}
Let $K$ be an algebraic number field and $\mathcal{O}_K$ be its ring of integers. Let $Int(\mathcal{O}_K)=\{f\in K[X] \mid f(\mathcal{O}_K)\subseteq\mathcal{O}_K\}$ be the ring of integer valued polynomials on $\mathcal{O}_K$. Then the number field $K$ is said to be a P\'olya field if the $\mathcal{O}_K$-module $Int(\mathcal{O}_K)$ has a regular basis, that is, a basis $(f_n)$ such that for each $n\in\mathbb{N}\cup\{0\}$, degree$(f_n)=n$ (see \cite{ZAN}). 
For each $n\in \mathbb{N}$, the leading coefficients of the polynomials in $Int(\mathcal{O}_K)$ of degree $n$ together with zero form a fractional ideal of $\mathcal{O}_K$, denoted by $\mathfrak{J}_n(K)$.
The following result establishes a connection between $\mathfrak{J}_n(K)$ and the P\'olya-ness of the number field $K$.
\begin{prop}\label{P1}
\cite{PJC} A number field $K$ is a P\'olya field if and only if $\mathfrak{J}_n(K)$ is principal for all integer $n\geq1$.
\end{prop}
It follows immediately from Proposition \ref{P1} that if the class number of $K$ is one then $K$ is a P\'olya field. However, the converse is not valid in general. That is, if the class number $h_K$ of $K$ is not one then we can not decide whether $K$ is a P\'olya field or not: for instance, every cyclotomic field is a P\'olya field (see \cite{ZAN}).

Let $Cl(K)$ denote the ideal class group of $K$. For each integer $n\geq1$, let $[\mathfrak{J}_n(K)]$ be the ideal class in $Cl(K)$ corresponding to the fractional ideal $\mathfrak{J}_n(K)$. The P\'olya group $Po(K)$ of $K$ is defined to be the subgroup of $Cl(K)$ generated by the elements $[\mathfrak{J}_n(K)]$ in $Cl(K)$. Therefore, $K$ is a P\'olya field if and only if $Po(K)=\{1\}$.

%

It is an interesting problem to look for  explicit families of number fields that are P\'olya /non-P\'olya (for example see \cite{BAH1},\cite{BAH2},\cite{AMA},\cite{CHA}). In fact, the classification of P\'olya fields of lower degree number fields is of significant interest. Towards this, Zantema \cite{ZAN} completely characterized quadratic P\'olya fields.
\begin{prop}
\cite{ZAN} Let $p$ and $q$ be two distinct odd primes. A quadratic field $\mathbb{Q}(\sqrt{d})$ is a P\'olya field if and only if $d$ is of one of the following forms :
\begin{enumerate}
    \item $d=2$, or $d=-1$, or $d=-2$, or $d=-p$ where $p\equiv3\pmod4$, or $d=p$,
    \item $d=2p$, or $d=pq$ where $pq\equiv1\pmod4$ and, in both cases, the fundamental unit has norm $1$ if $p\equiv1\pmod4$.
\end{enumerate}
\end{prop}
 Leriche \cite{AL2} completely classified Galois cubic P\'olya fields. We quote the precise result below.
 \begin{prop}
 \cite{AL2} Let $K$ be a cyclic cubic field. Then, $K$ is a P\'olya field if and only if $K=\mathbb{Q}(\theta)$ where $\theta$ is a root of a polynomial $P(X)$ such that:
 \begin{equation*}
     \text{either } P(X)=X^3-3X+1 \text{ or } P(X)=X^3-3pX-pu
 \end{equation*}
where $p$ is a prime such that $p=\frac{u^2+27w^2}{4}$ with $u\equiv2\pmod3$ and $w\in \mathbb{N}^*$.
 \end{prop}
 In the same article Leriche also classified cyclic quartic and cyclic sextic P\'olya fields. Moreover, he obtained similar classifications for some more families of bi-quadratic and sextic fields. (See \cite[Theorem 5.1, Theorem 6.2]{AL2} ). Recently, there have been several attempts to provide families of P\'olya/non-P\'olya fields in the remaining cases of the bi-quadratic extensions (see \cite{BAH1},\cite{BAH2},\cite{AMA},\cite{CHA}). In \cite{JAI} the authors constructed a new family of totally real bi-quadratic field with large P\'olya groups. 
 
In this article we characterize the P\'olya-ness of a special family of quintic fields arising from Lehmer qunitics. For each integer $n \in \mathbb{Z}$, the Lehmer quintic $f_n(x)\in \mathbb{Z}[x]$ is defined as 
\begin{align*}
    f_n(x) = & x^5+n^2x^4-(2n^3+6n^2+10n+10)x^3+(n^4+5n^3+11n^2+15n+5)x^2\\&+(n^3+4n^2+10n+10)x+1.
\end{align*}
Let $\theta_n\in\mathbb{C}$ be a root of $f_n(x)=0$. If we set $K_n=\mathbb{Q}(\theta_n)$ then $[K_n : \mathbb{Q}]=5$ and the fields $K_n$ are called Lehmer quintic fields \cite{LEH}. The main theorem we prove in this article is the following one:
\begin{thm}\label{NEW1}
Let $\{K_n\}$ be the family of Lehmer quintic fields. If $m_n=n^4+5n^3+15n^2+25n+25$ is cube free, then
\begin{enumerate}
    \item $K_n$ is a P\'olya field if and only if $m_n$ is a prime or $m_n=25$.
    \item We have $Po(K_n)\simeq(\mathbb{Z}/5\mathbb{Z})^{\omega({m_n})-1}$. Here $\omega(t)$ denotes the number of distinct prime divisors of $t$.
\end{enumerate}
 Moreover, there are infinitely many non-P\'olya fields in the family $K_n$. 
\end{thm}

Let $G$ be a finite group, and if $m>1$ is an integer then the $m$-rank of $G$ is defined to be the maximal integer $r$ such that $(\mathbb{Z}/m\mathbb{Z})^r$ is a subgroup of $G$. The following folklore conjecture is widely believed to be true but it is still open.
\begin{con}
Let $d>1$ and $m>1$ be two integers. Then the $m$-rank of the class group of $K$ is unbounded when $K$ runs through the number fields of degree $[K:\mathbb{Q}]=d$.
\end{con}
It is known that when $m=d$, or more generally when $m\mid d$, this conjecture follows from the class field theory (see \cite{JEA}). In the following corollary to \thmref{NEW1} we give an alternative and elementary proof of the conjecture for the case $m=d=5$.
\begin{cor}\label{cor1}
The 5-rank of the class groups of the non-P\'olya Lehmer quintic fields $K_n$ are unbounded.
\end{cor}

Leriche \cite{AL3} studied the embedding of a number field inside a P\'olya field and introduced the notion of P\'olya numbers. For a number field $K$ the minimum of degrees $[L:K]$ of extensions $L/K$ such that $L$ is a P\'olya field is called the P\'olya number of $K$ and it is denoted by $po_K$. We study the P\'olya numbers $po_{K_n}$ of non-P\'olya fields $K_n$ and obtained an upper bound of $po_{K_n}$ in terms of the size of its corresponding P\'olya groups $Po(K_n)$. Primarily we prove the following theorem.
\begin{thm}\label{NEW2}
Let $K_n$ be the family of Lehmer quintics fields such that $m_n$ is cube-free and $5\mid n$. Then $po_{K_n}\leq 5|Po(K_n)|$.
\end{thm}
Let $K$ be a number field and $\theta \in \mathcal{O}_K$ be a primitive element. The index $[\mathcal{O}_K:\mathbb{Z}[\theta]]$ is called the index of $\theta$ in $K$ and is denoted by $I(\theta)$. The index of the number field $K$ is defined as $I(K)=\gcd\{I(\theta)\mid \theta\in\mathcal{O}_K \mbox{ and }K=\mathbb{Q}(\theta)\}$. If $I(K)>1$ then the number field $K$ is not monogenic, that is, $\mathcal{O}_K\neq\mathbb{Z}[\theta]$ for any $\theta\in K$. However the the converse is not true in general. That is, there exists non-monogenic number fields $K$ with $I(K)=1$. These are basically fields which are not monogenic, but not for a local reason (see \cite{IND} for more details).  In this line we prove the following theorem in this article.
\begin{thm}\label{NEW3}
Let $K_n$ be the family of Lehmer quintic fields such that $m_n$ is cube-free and $5\mid n$. Then $K_n$ is not monogenic and $I(K_n)=1$.
\end{thm}
In Section 2, we develop some preliminaries required to prove \thmref{NEW1}. Section 3 contains the proof of \thmref{NEW1} and \corref{cor1}. In Section 4, we study the P\'olya numbers of Lehmer quintic fields and prove \thmref{NEW2}. In Section 5, we study the monogeneity of the non-P\'olya number fields $K_n$ and give the proof of \thmref{NEW3}. Finally in Section 7, we present some computations performed by us using SAGE.

\section{Preliminaries}

In this section we assume that the number field $K$ is a Galois extension of $\mathbb{Q}$. For any prime number $p$, the ramification index of $p$ in $K/\mathbb{Q}$ is denoted by $e_p$.  In \cite{JLC}, Chabert obtained a nice description for the cardinality of $Po(K)$ for cyclic extensions $K/\mathbb{Q}$.
\begin{prop}\label{CHABERT}
\cite{JLC} Assume that the extension $K/\mathbb{Q}$ is cyclic of degree $n$.
\begin{enumerate}
    \item If $K$ is real and $N(\mathcal{O}_K^{\times})=\{1\}$, then $|Po(K)|=\frac{1}{2n}\times\prod_p e_p$.
    \item In all other cases, $|Po(K)|=\frac{1}{n}\times\prod_p e_p$.
\end{enumerate}
\end{prop}

As an immediate corollary of Proposition \ref{CHABERT}, we record the following observation.
\begin{cor}\label{cor2}
If $K/\mathbb{Q}$ is cyclic extension of degree $n$ and $n$ is odd, then $|Po(K)|=\frac{1}{n}\times\prod_p e_p$.
\end{cor}

For a description of $Po(K)$ as in \cite{AL1}, we recall that an ideal class $[\mathfrak{a}]$ in $K$ is called an ambiguous ideal class if it is invariant under the action of $Gal(K/\mathbb{Q})$, that is, $$[\mathfrak{a}]^\tau=[\mathfrak{a}] \text{ for all }\tau\in Gal(K/\mathbb{Q}).$$ In \cite{AL1}, Leriche showed that for a Galois number field $K$, $Po(K)$ is the subgroup of $Cl(K)$ generated by the classes of the ambiguous ideals of $K$.

Now we state some results on Lehmer quintics. In \cite{RS}, Schoof and Washington showed that $f_n(x)$ is irreducible for all $n\in\mathbb{Z}$ and its discriminants equals to $(n^3+5n^2+10n+7)^2(n^4+5n^3+15n^2+25n+25)^4$. Let $\theta_n\in\mathbb{C}$ be a root of $f_n(x)=0$. If we set $K_n=\mathbb{Q}(\theta_n)$, then in \cite{RS}, it was shown that $K_n$ is cyclic field for all $n\in\mathbb{Z}$. We denote the ring of integers of $K_n$ by $\mathcal{O}_{K_n}$. Now we recall some results of Jeannin \cite{SJE} on the discriminant $d(K_n)$ of $K_n$.

\begin{lem}\label{lem1}
All the prime divisors $p \neq 5$ of $n^4+5n^3+15n^2+25n+25$ satisfy $p \equiv 1 \pmod 5$.
\end{lem}
Next lemma gives the discriminant of the field $K_n$.
\begin{lem}\label{lem2} The discriminant $d(K_n)$ is given by $d(K_n)=f(K_n)^4$, where the conductor $f(K_n)$ of $K_n$ is given by 
\begin{equation}\label{lc}
    f(K_n)=5^b\!\!\!\!\!\!\!\!\!\!\!\!\!\!\!\!\!\!\!\!\!\!\!\!\!\!\!\!\!\!\!\prod_{\substack{p\equiv1\!\!\!\!\!\!\pmod5\\ v_p(n^4+5n^3+15n^2+25n+25)\not\equiv0\!\!\!\!\!\pmod5}}\!\!\!\!\!\!\!\!\!\!\!\!\!\!\!\!\!\!\!\!\!\!\!\!\!\!\!\!\!\!\!p.
\end{equation}
Here $v_p(k)$ denotes the exponent of the highest power of the prime $p$ dividing a non-zero integer $k$ and 
\begin{equation}\label{b}
    b=
    \begin{cases}
    0, \text{ if } 5\nmid n,\\
    2, \text{ if } 5\mid n.
    \end{cases}
\end{equation}
\end{lem}
We quote the following result due to Erd\"os \cite{PER} which plays a crucial role in the proof of our main theorem.
\begin{thm}\label{ero}
\cite{PER} Let $f(x)$ be a polynomial of degree $l\geq3$ whose coefficients are integers with the highest common factor 1 and the leading coefficient is positive. Assume that $f(x)$ is not divisible by the $(l-1)$-th power of a linear polynomial with integer coefficient. Then there are infinitely many positive integers $n$ for which $f(n)$ is $(l-1)$-th power free.
\end{thm}

Now we state a deep result of on power-free values of polynomials (see \cite{HH1,HH2, TREUSS}).
\begin{thm}\label{Hel}
Let $f(x)\in\mathbb{Z}[x]$ be an irreducible polynomial of degree $d \geq3$ and assume that
$f$ has no fixed $(d-1)$th power prime divisor. Define
$$N_f'(X)=\#\{p\leq X : p \text{ prime}, f(p) \text{ is } (d-1)\text{-free}\}.$$
Then, for any $C>1$, we have 
$$N_f'(X)=c_f'\pi(X)+O_{C,f}\left(\frac{X}{(\log X)^C}\right),$$
as $x\rightarrow\infty$, where 
$$c_f'=\prod_{p}\left(1-\frac{\rho'(p^{d-1})}{\phi(p^{d-1})}\right)$$
and 
$$\rho'(d)=\#\{n\!\!\!\!\!\pmod d : (d,n)=1, d \mid f(n) \}.$$
\end{thm}

Let $f(x)$ be an irreducible polynomial with integral coefficients and $f(m)>0$ for $m=1,2,\dots$. Let $\omega(m)$ denote the number of distinct primes dividing $m$. Then for primes $p$, the following result due to Halberstam \cite{HAL} determines the distribution of values of $\omega(f(p))$.
\begin{thm}\label{ham} \cite{HAL}
Let $f(X) \in \mathbb{Z}[X]$ be any non-constant polynomial. For all but $o(X/ \log X)$ primes $p \leq X$,
\begin{equation*}
    \omega\Big(f(p)\Big)=\Big(1+o(1)\Big){\log \log n}.
\end{equation*}
\end{thm}

Now we state some results on the number of integral solutions of Diophantine equation of the type 
\begin{equation}\label{Pe2}
Y^m=f(X).
\end{equation} 
When $m=2$ and $f(x)$ is a monic qartic polynomial then the following result due to Masser \cite{MAS}, gives us the specific bound on the integral solutions to the curve.
\begin{thm}\label{mass}\cite{MAS}
Consider the diophantine equation 
$$Y^2 = f(X),$$
where $f(X)$ is a polynomial of degree four with integer coefficients. Assume that $f(X)$ is monic and not a perfect square. Then any integer solution $(x, y)$ of the above equation satisfies
$$|x| \leq 26 H(f)^3,$$
where $H(f)$ denotes the maximum of the absolute values of the coefficients of $f(X)$.
\end{thm}

\section{Proof of \thmref{NEW1} }

\begin{proof} We consider the set $$\mathcal{P}=\{n \in \mathbb{Z}: m_n=n^4+5n^3+15n^2+25n+25 \mbox{ is a cube-free integer} \}.$$  For $n \in \mathcal{P}$, we have
 \begin{equation}\label{e1}
   m_n=5^bAB^2.
 \end{equation}
 Here $b=0$ if $n$ is not divisible by $5$ and $b=2$ otherwise, and $A,B$ are square-free natural numbers which are relatively prime and $5 \nmid AB$.  Using Lemma \ref{lem1} and (\ref{lc}) we have 
 \begin{equation}
     f(K_n)=5^bAB \mbox{ and } d(K_n)=(5^bAB)^4.
 \end{equation}
 Since $K_n / \mathbb{Q}$ is Galois and of degree $5$, we see that for any prime $p$ the ramification index $e_p$ of $p$ in $K_n$ is given by
 \begin{align*}e_p=
 \begin{cases}
5 \mbox{ if } p\mid5AB,\\
1 \mbox{ otherwise}.
 \end{cases}
 \end{align*}
Thus, we obtain 
 \begin{equation}
     \prod_pe_p=5^{\omega(d(K_n))}=5^{\omega(m_n)},
 \end{equation}
Now from Corollary \ref{cor2} we get 
\begin{equation}\label{pol}
    |Po(K_n)|=5^{\omega(m_n)-1}.
\end{equation}

Thus, for $n \in \mathcal{P}$, the Lehmer quintic field $K_n$ is a P\'olya field if and only if $m_n$ is a prime or square of a prime. We claim that $m_n$ is a square of a prime if and only if $m_n=25$. This claim proves (1). To prove the claim, consider the curve 
\begin{equation}\label{curve}
 Y^2=f(X)=X^4+5X^3+15X^2+25X+25.   
\end{equation}
From \thmref{mass}, we see that any integral solution $(x,y)$ of (\ref{curve}) satisfies 
$$|x|\leq 26\times25^3=406250.$$
Using a SAGE program we find that for $x\in[-406250,406250]$ the only integral solution to the curve $Y^2=f(X)$ is $(x,y)=(0,5)$. In other words $m_n=n^4+5n^3+15n^2+25n+25$ is not a square for any non-zero integer $n$ unless $m_n=25$.  This establishes the claim.


We have $Gal(K_n/\mathbb{Q})\simeq \mathbb{Z}/5\mathbb{Z}$. Let $\sigma$ be a generator of $Gal(K_n/\mathbb{Q})$ and $[\mathfrak{I}]\neq [1]$ be an ambiguous ideal class in $K_n$. Then we have,
\begin{align*}
    [\mathfrak{I}]^5&=[\mathfrak{I}][\mathfrak{I}][\mathfrak{I}][\mathfrak{I}][\mathfrak{I}]\\    
    &=[\mathfrak{I}][\mathfrak{I}]^\sigma[\mathfrak{I}]^{\sigma^2}[\mathfrak{I}]^{\sigma^3}[\mathfrak{I}]^{\sigma^4}\\
    &=[N(\mathfrak{I})]\\
    &=[1]
\end{align*}
where $N(\mathfrak{I})\in\mathbb{Q}$ denotes the norm of the ideal $\mathfrak{I}$. We conclude that the order of any non-trivial ambiguous ideal class in the class group of $K_n$ is $5$. From the structure theorem for the abelian groups we obtain $$Po(K_n)\simeq\left(\frac{\mathbb{Z}}{5\mathbb{Z}}\right)^{\omega({m_n})-1}.$$
This finishes the proof of $(2)$ of \thmref{NEW1}.

From (\ref{e1}), we see that $m_n$ can not be a prime whenever $5|n$. Thus $K_n$ is a non-P\'olya field whenever $n \neq 0, 5|n$ and $n \in \mathcal{P}$. Next we show that there are infinitely many such $n'$s. To do this we show that there are infinitely many $k$ such that $m_{5k}$ is cube-free. Put
\begin{equation}
\begin{aligned}
      m_{5k}:&=(5k)^4+5(5k)^3+15(5k)^2+25(5k)+25\\  
          &=25(25k^4+25k^3+15k^2+5k+1)\\
          &=25g(k).
\end{aligned}
\end{equation}
If $h(k)=ak+b$ is a linear polynomial such that $h(k)^3\mid g(k)$, then for $t=-b/a\in\mathbb{Q}$, we have 
\begin{align}
    g(t)&=0\implies 25t^4+25t^3+15t^2+5t+1=0,\\
    g'(t)&=0\implies 100t^3+75t^2+30t+5=0,\\
    g''(t)&=0\implies 300t^2+150t+30=0.
\end{align}
This contradicts the fact that $t\in\mathbb{Q}$. Thus, from Theorem \ref{ero}, it follows that $g(k)$ is cube-free for infinitely many $k$. Since $5\nmid g(k)$ for all $k$, it follows that $m_{5k}$ is cube-free for infinitely many integers $k$. This proves that $\mathcal{P}$ is an infinite set and completes the proof of the theorem.
\end{proof}

\begin{rmk}
From the proof of Theorem \ref{NEW1} it follows that for any $n \neq 0$, the Lehmer quintic field $K_{5n}$ is non-P\'olya whenever $m_{5n}$ is cube-free. However, there are non-P\'olya fields $K_{5n}$ with $m_{5n}$ not being cube-free (see the entry for $n=-53$ in the Table 1).\\

Conjecturally, there are infinitely many $n \in \mathbb{Z}$ such that $m_n$ is prime and thus the family $K_n$ should have infinitely many P\'olya fields. Under the assumption that $m_n$ is cube-free, Theorem \ref{NEW1} asserts that there are infinitely many P\'olya fields in the family $K_n$ only if there are infinitely many primes of the form $m_n$.
\end{rmk}

\begin{proof}
(Proof of \corref{cor1}) From the above remark, it is enough to find integers $n$ such that $m_{5n}$ is cube-free and $\omega(m_{5n})$ goes to infinity as $n$ goes to infinity. Let $g(k)=(25k^4+25k^3+15k^2+5k+1)$, then from Theorem \ref{Hel} for a positive proportion of primes $p$ we have $g(p)$ is cube-free. Consequently $m_{5p}$ is cube-free for a positive proportion of prime numbers $p$. Now onwards, we only consider those primes $p$ such that $m_{5p}$ is cube-free.  There is a positive constant $c$ such that for any large real number $X$ there are at least $c \frac{X}{\log X}$ many primes.  Now from Theorem \ref{ham}, it follows that $\omega(m_{5p})$ goes to infinity as $p$ goes to infinity.

%
\end{proof}

\section{P\'olya Numbers and monogenity of Lehmer Quintic Fields}

Recall the classical embedding problem: Is every number field $K$ contained in a field $L$ with class number one? In 1964, Golod and Shafarevich \cite{ESG} gave a negative answer to this question. The corresponding embedding problem for P\'olya fields was confirmed affirmatively by Leriche \cite{AL3}.  Leriche proved that every number field is contained in a P\'olya field, namely its Hilbert class field.

A minimal P\'olya field over $K$ is a field extension $L$ of $K$ which is a P\'olya field and such that no intermediate field $K\subseteq M\subsetneq L$ is a P\'olya field. 
\begin{defn} \cite{AL3}
The P\'olya number of $K$ is given by the integer 
$$po_K=\text{min}\{[L : K]\mid K\subseteq L,\: L\text{ P\'olya field}\}.$$
\end{defn}

The genus field (resp. genus field in the narrow sense) of $K$ is the maximal abelian extension $\Gamma_K$(resp. $\Gamma_K'$) of $K$ which is a compositum of $K$ with an absolute abelian number field and is unramified over $K$ at all places (resp. all finite places) of $K$. The genus number of $K$ is defined to be the degree $g_K=[\Gamma_K : K]$. If $K$ is abelian then Leriche showed that the genus field $\Gamma_K$ is P\'olya and hence
\begin{equation}
po_K\leq g_K.
\end{equation}

On the other hand, Zantema proved that both the cyclotomic and real cyclotomic fields are P\'olya fields \cite{ZAN}. Thus for abelian number fields $K$, if $f$ is the conductor of $K$ and $\phi(f)$ is the value of Euler totient function then we have
\begin{equation}
po_K\leq\frac{\phi(f)}{[K : \mathbb{Q}]}\quad\text{  and   }\quad po_K\leq\frac{\phi(f)}{2[K : \mathbb{Q}]}\text{ if $K$ is real}.
\end{equation}

For the general case, when $K$ is a Galois number field (not necessarily abelian) with class number $h_K$ then we have 
\begin{equation}
po_K\leq h_K.
\end{equation}

To prove Theorem \thmref{NEW2} we need the following result due to Ishida \cite{MI} on the genus number of a cyclic number field of prime degree.
\begin{thm}\label{polno}\cite{MI}
Let $K$ be a cyclic number field of degree $q$, where $q$ is an odd prime. If $t$ is the number of primes $p$ such that $p$ is totally ramified in $K$. Then the genus number $g_K$ of $K$ is given as follows.
$$g_K=\begin{cases}
q^t, \text{ if } q \text{ is totally ramified in }$K$,\\
q^{t-1}, \text{ otherwise. }
\end{cases}$$
\end{thm}

\begin{proof} (Proof of \thmref{NEW2})
We have already seen in the proof of Theorem \ref{NEW1} that, the number of primes $p$ such that $p$ is totally ramified in $K_n$ is $\omega(m_n)$. Now applying Theorem \ref{polno} to the family of number fields $K_n$ we get,
\begin{align}\label{polyano}
   g_{K_n}=5^{\omega(m_n)}
   \implies po_{K_n}\leq 5^{\omega(m_n)}.
   \end{align}
Using (\ref{pol}) in (\ref{polyano}) we obtain
\begin{equation}
    po_{K_n}\leq 5|Po(K_n)|.
\end{equation}
\end{proof} 
\begin{rmk}
Generally both $po_{K}$ and $Po(K)$ are mutually independent objects, but here in the case of non-P\'olya Lehmer quintic fields we have an unexpected relation.
\end{rmk}

To prove Theorem \thmref{NEW3} we need the following result of Gras \cite{MNG} on the monogenicity of cyclic number fields of prime degree.
\begin{prop}\label{gras}
\cite{MNG} If $K$ is a cyclic number field of prime degree $p\geq5$ then $K$ is monogenic only if $2p+1$ is prime and it is the maximal real subfield of cyclotomic field $\mathbb{Q}(\zeta_{2p+1})$.
\end{prop}
Lastly, we recall a result of Zylinski \cite{ZAL}.
\begin{prop}
If $K$ is a number field of degree $n$, then $I(K)$ has only prime divisors $p$ satisfying $p<n$.
\end{prop}

\begin{proof} (Proof of \thmref{NEW3})
Let $K_n$ be the family of non-P\'olya Lehmer quintics fields. We know that $Gal(K_n/\mathbb{Q})\simeq\mathbb{Z}/5\mathbb{Z}$. From Theorem \ref{NEW1} we have
\begin{equation}
    |Po(K_n)|\geq5.
\end{equation}
Since real cyclotomic fields are P\'olya fields, $K_n$ never occurs as the maximal real subfield of a cyclotomic field. From Proposition \ref{gras} it follows that $K_n$ is not monogenic.

Next, we aim to show $I(K_n)=1$, for all non-zero $n$ that are divisible by $5$.  We recall the relation
\begin{equation}
d(\theta_n)=[I(\theta_n)]^2d(K_n)
\end{equation}
As mentioned earlier,
$$d(\theta_n)=(n^3+5n^2+10n+7)^2(n^4+5n^3+15n^2+25n+25)^4.$$
From Lemma \ref{lem1}, we see that $(n^4+5n^3+15n^2+25n+25)$ is not divisible by $2$ or $3$.  It is easily seen that $(n^3+5n^2+10n+7)$ is also not divisible by $2$ or $3$. Consequently, we conclude that $I(\theta_n)$ is not divisible by $2$ or $3$. Now, from the result of Zylinski, it follows that $I(\theta_n)=1$.


\end{proof}

\section{Computation}
In this section we present our computations showing that there are many non-P\'olya fields in the family $K_{5n}$. For most of $n \in \{-60,-59, \ldots, 59,60\}$ we see that $m_{5n}$ is cube-free with only exception occurring at $n=-53$.  For all $n$ in this range the P\'olya group is non-trivial. In fact, for $n=-53$, $m_{5n}$ is not cube-free but the field $K_{5n}$ is non-P\'olya. In the table below $C_{m_{5n}}$ denotes the cube part of $m_{5n}$.These computations are performed by us using SAGE.

\begin{table}[H]
\caption{Family of non-P\'olya fields}
\parbox[t]{0.49\linewidth}{
\centering
\begin{longtable}[t]{c c c c}\hline
$n$ & $m_{5n}$& \Gape[.15cm][.15cm] {$C_{m_{5n}}$}  & $\#Po(K_{5n})$ \\
\hline
-60 & 7966342525 & 1 & $5^ 2 $ \\
-59 & 7446286775 & 1 & $5^ 2 $ \\
-58 & 6952119275 & 1 & $5^ 2 $ \\
-57 & 6482966275 & 1 & $5^ 2 $ \\
-56 & 6037969025 & 1 & $5^ 3 $ \\
-55 & 5616283775 & 1 & $5^ 2 $ \\
-54 & 5217081775 & 1 & $5^ 3 $ \\
\textbf{-53} & \textbf{4839549275} & \textbf{1331} & \textbf{5}\textsuperscript{\textbf{2}} \\
-52 & 4482887525 & 1 & $5^ 3 $ \\
-51 & 4146312775 & 1 & $5^ 4 $ \\
-50 & 3829056275 & 1 & $5^ 3 $ \\
-49 & 3530364275 & 1 & $5^ 2 $ \\
-48 & 3249498025 & 1 & $5^ 1 $ \\
-47 & 2985733775 & 1 & $5^ 3 $ \\
-46 & 2738362775 & 1 & $5^ 2 $ \\
-45 & 2506691275 & 1 & $5^ 2 $ \\
-44 & 2290040525 & 1 & $5^ 2 $ \\
-43 & 2087746775 & 1 & $5^ 2 $ \\
-42 & 1899161275 & 1 & $5^ 2 $ \\
-41 & 1723650275 & 1 & $5^ 2 $ \\
-40 & 1560595025 & 1 & $5^ 4 $ \\
-39 & 1409391775 & 1 & $5^ 2 $ \\
-38 & 1269451775 & 1 & $5^ 2 $ \\
-37 & 1140201275 & 1 & $5^ 1 $ \\
-36 & 1021081525 & 1 & $5^ 2 $ \\
-35 & 911548775 & 1 & $5^ 2 $ \\
-34 & 811074275 & 1 & $5^ 2 $ \\
-33 & 719144275 & 1 & $5^ 1 $ \\
-32 & 635260025 & 1 & $5^ 1 $ \\
-31 & 558937775 & 1 & $5^ 3 $ \\
-30 & 489708775 & 1 & $5^ 2 $ \\
-29 & 427119275 & 1 & $5^ 3 $ \\
-28 & 370730525 & 1 & $5^ 2 $ \\
-27 & 320118775 & 1 & $5^ 2 $ \\
-26 & 274875275 & 1 & $5^ 2 $ \\
-25 & 234606275 & 1 & $5^ 1 $ \\
-24 & 198933025 & 1 & $5^ 3 $ \\
-23 & 167491775 & 1 & $5^ 3 $ \\
-22 & 139933775 & 1 & $5^ 1 $ \\
-21 & 115925275 & 1 & $5^ 3 $ \\
-20 & 95147525 & 1 & $5^ 3 $ \\

\hline
\end{longtable}
	
}
\hfill
\parbox[t]{0.50\linewidth}{
\centering
\begin{longtable}[t]{c c c c}
\hline
$n$ & $m_{5n}$& \Gape[.15cm][.15cm] {$C_{m_{5n}}$} & $\#Po(K_{5n})$ \\
\hline
1 & 1775 & 1 & $5^ 1 $ \\
2 & 16775 & 1 & $5^ 2 $ \\
3 & 71275 & 1 & $5^ 1 $ \\
4 & 206525 & 1 & $5^ 2 $ \\
5 & 478775 & 1 & $5^ 2 $ \\
6 & 959275 & 1 & $5^ 1 $ \\
7 & 1734275 & 1 & $5^ 1 $ \\
8 & 2905025 & 1 & $5^ 1 $ \\
9 & 4587775 & 1 & $5^ 1 $ \\
10 & 6913775 & 1 & $5^ 3 $ \\
11 & 10029275 & 1 & $5^ 2 $ \\
12 & 14095525 & 1 & $5^ 1 $ \\
13 & 19288775 & 1 & $5^ 2 $ \\
14 & 25800275 & 1 & $5^ 2 $ \\
15 & 33836275 & 1 & $5^ 3 $ \\
16 & 43618025 & 1 & $5^ 2 $ \\
17 & 55381775 & 1 & $5^ 3 $ \\
18 & 69378775 & 1 & $5^ 2 $ \\
19 & 85875275 & 1 & $5^ 1 $ \\
20 & 105152525 & 1 & $5^ 1 $ \\
21 & 127506775 & 1 & $5^ 3 $ \\
22 & 153249275 & 1 & $5^ 2 $ \\
23 & 182706275 & 1 & $5^ 1 $ \\
24 & 216219025 & 1 & $5^ 2 $ \\
25 & 254143775 & 1 & $5^ 3 $ \\
26 & 296851775 & 1 & $5^ 2 $ \\
27 & 344729275 & 1 & $5^ 3 $ \\
28 & 398177525 & 1 & $5^ 2 $ \\
29 & 457612775 & 1 & $5^ 1 $ \\
30 & 523466275 & 1 & $5^ 2 $ \\
31 & 596184275 & 1 & $5^ 2 $ \\
32 & 676228025 & 1 & $5^ 2 $ \\
33 & 764073775 & 1 & $5^ 1 $ \\
34 & 860212775 & 1 & $5^ 2 $ \\
35 & 965151275 & 1 & $5^ 3 $ \\
36 & 1079410525 & 1 & $5^ 2 $ \\
37 & 1203526775 & 1 & $5^ 2 $ \\
38 & 1338051275 & 1 & $5^ 3 $ \\
39 & 1483550275 & 1 & $5^ 2 $ \\
40 & 1640605025 & 1 & $5^ 2 $ \\
41 & 1809811775 & 1 & $5^ 2 $ \\
\hline

\end{longtable}
\qquad\qquad(To be continued)
}
	\end{table}

\begin{table}[H]
\parbox[t]{0.49\linewidth}{
\begin{longtable}[t]{c c c c}\hline
$n$ & $m_{5n}$& \Gape[.15cm][.15cm] {$C_{m_{5n}}$}  & $\#Po(K_{5n})$ \\
\hline

-19 & 77296775 & 1 & $5^ 2 $ \\
-18 & 62084275 & 1 & $5^ 3 $ \\
-17 & 49236275 & 1 & $5^ 2 $ \\
-16 & 38494025 & 1 & $5^ 2 $ \\
-15 & 29613775 & 1 & $5^ 1 $ \\
-14 & 22366775 & 1 & $5^ 2 $ \\
-13 & 16539275 & 1 & $5^ 2 $ \\
-12 & 11932525 & 1 & $5^ 2 $ \\
-11 & 8362775 & 1 & $5^ 1 $ \\
-10 & 5661275 & 1 & $5^ 1 $ \\
-9 & 3674275 & 1 & $5^ 3 $ \\
-8 & 2263025 & 1 & $5^ 2 $ \\
-7 & 1303775 & 1 & $5^ 2 $ \\
-6 & 687775 & 1 & $5^ 3 $ \\
-5 & 321275 & 1 & $5^ 2 $ \\
-4 & 125525 & 1 & $5^ 1 $ \\
-3 & 36775 & 1 & $5^ 1 $ \\
-2 & 6275 & 1 & $5^ 1 $ \\
-1 & 275 & 1 & $5^ 1 $ \\
\hline
\end{longtable}
	
}
\hfill
\parbox[t]{0.50\linewidth}{
\centering
\begin{longtable}[t]{c c c c}
\hline
$n$ & $m_{5n}$& \Gape[.15cm][.15cm] {$C_{m_{5n}}$} & $\#Po(K_{5n})$ \\
\hline

42 & 1991781775 & 1 & $5^ 3 $ \\
43 & 2187141275 & 1 & $5^ 4 $ \\
44 & 2396531525 & 1 & $5^ 1 $ \\
45 & 2620608775 & 1 & $5^ 2 $ \\
46 & 2860044275 & 1 & $5^ 2 $ \\
47 & 3115524275 & 1 & $5^ 3 $ \\
48 & 3387750025 & 1 & $5^ 3 $ \\
49 & 3677437775 & 1 & $5^ 4 $ \\
50 & 3985318775 & 1 & $5^ 1 $ \\
51 & 4312139275 & 1 & $5^ 1 $ \\
52 & 4658660525 & 1 & $5^ 2 $ \\
53 & 5025658775 & 1 & $5^ 2 $ \\
54 & 5413925275 & 1 & $5^ 2 $ \\
55 & 5824266275 & 1 & $5^ 3 $ \\
56 & 6257503025 & 1 & $5^ 3 $ \\
57 & 6714471775 & 1 & $5^ 3 $ \\
58 & 7196023775 & 1 & $5^ 2 $ \\
59 & 7703025275 & 1 & $5^ 3 $ \\
60 & 8236357525 & 1 & $5^ 2 $ \\
\hline
\end{longtable}
}
	\end{table}


\end{document}